\documentclass[12pt,a4paper]{article}
\usepackage{mathrsfs}
\usepackage{indentfirst}
\usepackage{amsmath,amssymb,amsfonts,amsthm,graphics,graphicx}
\usepackage{makeidx}
\usepackage{color}

\newtheorem{theorem}{Theorem}

\newtheorem{lemma}{Lemma}

\newtheorem{conjecture}{Conjecture}
\begin{document}
\title{\Large\bf Further hardness results on the\\ generalized
connectivity of graphs\footnote{Supported by NSFC No.11071130 and
the ``973" program.}}
\author{\small Lily Chen, Xueliang Li, Mengmeng Liu, Yaping Mao
\\
\small Center for Combinatorics and LPMC-TJKLC
\\
\small Nankai University, Tianjin 300071, China
\\
\small lily60612@126.com; lxl@nankai.edu.cn;\\
\small liumm05@163.com; maoyaping@ymail.com.}
\date{}
\maketitle
\begin{abstract}
The generalized $k$-connectivity $\kappa_k(G)$ of a graph $G$ was
introduced by Chartrand et al. in 1984, which is a nice
generalization of the classical connectivity. Recently, as a natural
counterpart, Li et al. proposed the concept of generalized
edge-connectivity for a graph. In this paper, we determine the
computational complexity of the generalized connectivity and
generalized edge-connectivity of a graph. Two conjectures are also
proved to be true.

\end{abstract}

\section{Introduction}

All graphs considered in this paper are undirected, finite and
simple. We refer to the book \cite{bondy} for graph theoretical
notation and terminology not described here. The generalized
connectivity of a graph $G$, introduced by Chartrand et al. in
\cite{Chartrand1}, is a natural and nice generalization of the
concept of the standard (vertex-)connectivity. For a graph $G(V,E)$
and a set $S\subseteq V(G)$ of at least two vertices, \emph{an
$S$-Steiner tree} or \emph{a Steiner tree connecting $S$} (or
simply, \emph{an $S$-tree}) is such a subgraph $T(V',E')$ of $G$
that is a tree with $S\subseteq V'$. Two Steiner trees $T$ and $T'$
connecting $S$ are said to be \emph{internally disjoint} if
$E(T)\cap E(T')=\varnothing$ and $V(T)\cap V(T')=S$. For $S\subseteq
V(G)$ and $|S|\geq 2$, the \emph{generalized local connectivity}
$\kappa(S)$ is the maximum number of internally disjoint trees
connecting $S$ in $G$. Note that when $|S|=2$ a Steiner tree
connecting $S$ is just a path connecting $S$. For an integer $k$
with $2\leq k\leq n$, the \emph{generalized $k$-connectivity}
$\kappa_k(G)$ of $G$ is defined as $\kappa_k(G)= min\{\kappa(S) :
S\subseteq V(G) \ and \ |S|=k\}$. Clearly, when $|S|=2$,
$\kappa_2(G)$ is nothing new but the connectivity $\kappa(G)$ of
$G$, that is, $\kappa_2(G)=\kappa(G)$, which is the reason why one
addresses $\kappa_k(G)$ as the generalized connectivity of $G$.  Set
$\kappa_k(G)=0$ when $G$ is disconnected. Results on the generalized
connectivity can be found in \cite{Chartrand2,
LLSun, LLL1, LLL2, LL, LLZ, Okamoto}.

As a natural counterpart of the generalized connectivity, Li et al.
introduced the concept of generalized edge-connectivity in
\cite{LMS}. For $S\subseteq V(G)$ and $|S|\geq 2$, the
\emph{generalized local edge-connectivity} $\lambda(S)$ is the
maximum number of edge-disjoint trees connecting $S$ in $G$. For an
integer $k$ with $2\leq k\leq n$, the \emph{generalized
$k$-edge-connectivity} $\lambda_k(G)$ of $G$ is then defined as
$\lambda_k(G)= min\{\lambda(S) : S\subseteq V(G) \ and \ |S|=k\}$.
It is also clear that when $|S|=2$, $\lambda_2(G)$ is nothing new
but the standard edge-connectivity $\lambda(G)$ of $G$, that is,
$\lambda_2(G)=\lambda(G)$, which is the reason why we address
$\lambda_k(G)$ as the generalized edge-connectivity of $G$. Also set
$\lambda_k(G)=0$ when $G$ is disconnected.

The generalized edge-connectivity is related to an important
problem, which is called the \emph{Steiner Tree Packing Problem}.
For a given graph $G$ and $S\subseteq V(G)$, this problem asks to
find a set of maximum number of edge-disjoint Steiner trees
connecting $S$ in $G$. One can see that the Steiner Tree Packing
Problem studies local properties of graphs, but the generalized
edge-connectivity focuses on global properties of graphs. The
generalized edge-connectivity and the Steiner Tree Packing Problem
have applications in $VLSI$ circuit design, see \cite{Grotschel1,
Grotschel2, Jain, Sherwani}. In this application, a Steiner tree is
needed to share an electronic signal by a set of terminal nodes.
Another application, which is our primary focus, arises in the
Internet Domain. Imagine that a given graph $G$ represents a
network. We choose arbitrary $k$ vertices as nodes. Suppose that one
of the nodes in $G$ is a \emph{broadcaster}, and all the other nodes
are either \emph{users} or \emph{routers} (also called
\emph{switches}). The broadcaster wants to broadcast as many streams
of movies as possible, so that the users have the maximum number of
choices. Each stream of movie is broadcasted via a tree connecting
all the users and the broadcaster. So, in essence we need to find
the maximum number Steiner trees connecting all the users and the
broadcaster, namely, we want to get $\lambda (S)$, where $S$ is the
set of the $k$ nodes. Clearly, it is a Steiner tree packing problem.
Furthermore, if we want to know whether for any $k$ nodes the
network $G$ has the above properties, then we need to compute
$\lambda_k(G)$ in order to prescribe the reliability and the
security of the network.

As we know, for any graph $G$, we have polynomial-time algorithms to
get the connectivity $\kappa(G)$ and the edge-connectivity
$\lambda(G)$. A natural question is whether there is a
polynomial-time algorithm to get the $\kappa_3(G)$ or
$\lambda_3(G)$, or more generally $\kappa_k(G)$ or $\lambda_k(G)$.

In \cite{LLZ}, the authors described a polynomial-time algorithm to
decide whether $\kappa_3(G)\geq \ell$.

\begin{theorem}\cite{LLZ}\label{th1}
Given a fixed positive integer $\ell$, for any graph $G$ the problem
of deciding whether $\kappa_3(G)\geq \ell$ can be solved by a
polynomial-time algorithm.
\end{theorem}

As a continuation of their investigation, S. Li and X. Li later
turned their attention to the general $\kappa_k$ and obtained the
following results in \cite{LL}.

\begin{theorem}\cite{LL}\label{th2}
For two fixed positive integers $k$ and $\ell$, given a graph $G$, a $k$-subset $S$
of $V(G)$, the problem of deciding whether
there are $\ell$ internally disjoint trees connecting $S$ can be solved
by a polynomial-time algorithm.
\end{theorem}

\begin{theorem}\cite{LL}\label{th3}
For any fixed integer $k\geq 4$, given a graph $G$, a $k$-subset $S$
of $V(G)$ and an integer $ \ell \ (2\leq \ell\leq n-2)$, deciding whether
there are $\ell$ internally disjoint trees connecting $S$, namely
deciding whether $\kappa(S)\geq \ell$, is $\mathcal {N}\mathcal
{P}$-complete.
\end{theorem}

\begin{theorem}\cite{LL}\label{th4}
For any fixed integer $\ell\geq 2$, given a graph $G$ and a subset $S$
of $V(G)$, deciding whether there are $\ell$ internally disjoint trees connecting $S$, namely
deciding whether $\kappa(S)\geq \ell$, is $\mathcal {N}\mathcal
{P}$-complete.
\end{theorem}

In Theorem \ref{th3}, for $k=3$ the complexity problem was not
solved in \cite{LL}, and the problem is still open. So, S. Li in her
Ph.D. thesis \cite{SLi} conjectured that it is $\mathcal {N}\mathcal
{P}$-complete.

\begin{conjecture} \cite{SLi}\label{con1}
Given a graph $G$ and a $3$-subset $S$ of $V(G)$ and an integer
$ \ell \ (2\leq \ell\leq n-2)$, deciding whether there are $\ell$ internally
disjoint trees connecting $S$, namely deciding whether
$\kappa(S)\geq \ell$, is $\mathcal {N}\mathcal {P}$-complete.
\end{conjecture}

Since $\kappa_k(G)=min\{\kappa(S)\}$, where the minimum is taken
over all $k$-subsets $S$ of $V(G)$, S. Li also considered the
complexity of the problem of deciding whether $\kappa_k(G)\geq
\ell$, and conjectured that it is $\mathcal
{N}\mathcal{P}$-complete.

\begin{conjecture} \cite{SLi}\label{con2}
For a fixed integer $k\geq 3$, given a graph $G$ and an integer
$ \ell \ (2\leq \ell\leq n-2)$, the problem of deciding whether
$\kappa_k(G)\geq \ell$, is $\mathcal {N}\mathcal {P}$-complete.
\end{conjecture}

In this paper, we will confirm that these two conjectures are true.

For the generalized $k$-edge-connectivity $\lambda_k(G)$, it is also
natural to consider its computational complexity problem: for any
two positive integers $k$ and $\ell$, given a $k$-subset $S$ of
$V(G)$, is there a polynomial-time algorithm to determine whether
$\lambda(S)\geq \ell$ ?

If both $k$ and $\ell$ are fixed integers, we will reduce it to the
problem in Theorem \ref{th2}, and prove that there is a
polynomial-time algorithm to determine whether $\lambda_k(G)\geq
\ell$. If one of $k$ and $\ell$ is not fixed, then the problem turns
out to be $\mathcal {N}\mathcal {P}$-complete.

The rest of this paper is organized as follows. In next section we
give the proofs of Conjectures \ref{con1} and \ref{con2}. Section
$3$ contains the hardness results of the generalized
edge-connectivity.

\section{Proofs of the two conjectures}

In this section, we focus on solving Conjectures \ref{con1} and
\ref{con2}. In order to show that these conjectures are correct, we
first introduce a basic $\mathcal {N}\mathcal {P}$-complete problem
and a new problem.

\noindent\textbf{3-DIMENSIONAL MATCHING (3-DM)}: Given three sets
$U,\ V$ and $W$ with $|U|=|V|=|W|$, and a subset $T$ of $U\times
V\times W$, decide whether there is a subset $M$ of $T$ with
$|M|=|U|$ such that whenever $(u,v,w)$ and $(u',v',w')$ are distinct
triples in $M$, we have $u\neq u',$ $v\neq v',$ and $w\neq w'$ ?

\noindent\textbf{Problem 1:} Given a tripartite graph $G=(V,E)$ with
three partitions $(\overline{U},\overline{V},\overline{W})$, and
$|\overline{U}|=|\overline{V}|=|\overline{W}|=q$, decide whether
there is a partition of $V$ into $q$ disjoint $3$-sets
$V_1,V_2,\ldots, V_q$ such that every
$V_i=\{v_{i_1},v_{i_2},v_{i_3}\}$ satisfies that $v_{i_1}\in
\overline{U}$, $v_{i_2}\in \overline{V}$, $v_{i_3}\in \overline{W}
$, and $G[V_i]$ is connected ?

By reducing $3$-DM to Problem $1$, we can get the following result.

\begin{lemma}\label{lem1}
Problem 1 is $\mathcal {N}\mathcal {P}$-complete.
\end{lemma}

\begin{proof}
It is easy to see that Problem $1$ is in $\mathcal {N}\mathcal {P}$
since given a partition of $V(G)$ into $q$ disjoint $3$-sets $V_i\
(i=1,2,\ldots, q)$, one can check in polynomial time that every
$V_i=\{v_{i_1},v_{i_2},v_{i_3}\}$ satisfies that $v_{i_1}\in
\overline{U} $, $v_{i_2}\in \overline{V}$, $v_{i_3}\in \overline{W}
$, and $G[V_i]$ is connected.

We now prove that $3$-DM is polynomially reducible to this problem.

Given three sets $U,\ V$ and $W$ with $|U|=|V|=|W|=n$, and a subset
$T$ of $U\times V\times W$. Let $T=\{T_1,T_2,\ldots,T_m\}$. We will
construct a tripartite graph
$G[\overline{U},\overline{V},\overline{W}]$ with
$|\overline{U}|=|\overline{V}|=|\overline{W}|=q$ such that the
desired partition exists for $G$ if and only if there is a subset
$M$ of $T$ with $|M|=n$ and whenever $(u,v,w)$ and $(u',v',w')$ are
distinct triples in $M$, we have $u\neq u',$ $v\neq v'$ and $w\neq
w'$.

For each $T_i=(u_i,v_i,w_i)$, we add $18$ new vertices
$V_i^0=\{t_{i_1},t_{i_2},\ldots, t_{i_{18}}\}$ and $26$ edges
$E_i^0$ which are shown in Figure $1$. Thus $G[\overline{U},\overline{V},\overline{W}]$ is
defined by
 $$\overline{U}=U\bigcup \{t_{i_3},t_{i_6},t_{i_7},t_{i_{10}},t_{i_{13}}, t_{i_{16}}: 1\leq i \leq m\}$$
 $$\overline{V}=V\bigcup \{t_{i_1},t_{i_4},t_{i_9},t_{i_{12}},t_{i_{14}}, t_{i_{17}}: 1\leq i \leq m\}$$
 $$\overline{W}=W\bigcup \{t_{i_2},t_{i_5},t_{i_8},t_{i_{11}},t_{i_{15}}, t_{i_{18}}: 1\leq i \leq m\}$$
 $$V=\overline{U} \cup \overline{V}\cup\overline{W},\  E=\bigcup_{i=1}^{m}E_i^0.$$

\begin{figure}[h,t,b,p]
\begin{center}
\scalebox{0.9}[0.9]{\includegraphics{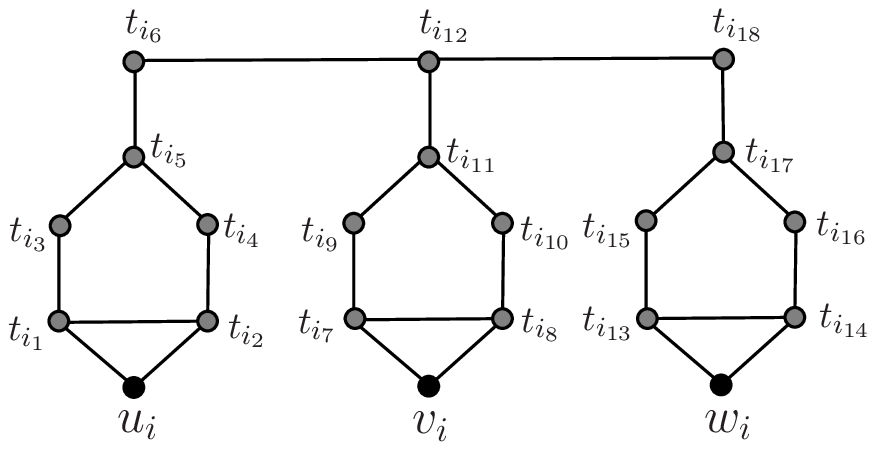}}\\
Figure 1. Graphs for Lemma \ref{lem1}.
\end{center}
\end{figure}

Note that $|V|=3n+18m,$ $|E|=26m$. Thus this instance can be
constructed in polynomial time from a $3$-DM instance. Now that
$q=n+6m$.

If there is a subset $M$ of $T$ with $|M|=n$, and whenever $(u,v,w)$
and $(u',v',w')$ are distinct triples in $M$ we have $u\neq u',$
$v\neq v',$ and $w\neq w'$, then the corresponding partition
$V=V_1\cup V_2\cup \ldots \cup V_q$ is given by taking
$\{u_i,t_{i_1},t_{i_2}\}, \{v_i,t_{i_7},t_{i_8}\},
\{w_i,t_{i_{13}},t_{i_{14}}\},\{t_{i_3},t_{i_{4}},t_{i_{5}}\},$
$\{t_{i_9},t_{i_{10}},t_{i_{11}}\},\{t_{i_{15}},$
$t_{i_{16}},t_{i_{17}}\},$ $\{t_{i_6},t_{i_{12}},t_{i_{18}}\}$ from
the vertices of $V_i^0\cup T_i$ whenever $T_i=(u_i,v_i,w_i)$ is in
$M$, and by taking $\{t_{i_1},t_{i_2},t_{i_3}\},$ $
\{t_{i_4},t_{i_5},t_{i_6}\},\ \{t_{i_7},t_{i_8},t_{i_9}\},$
$\{t_{i_{10}},t_{i_{11}},t_{i_{12}}\},\ \{t_{i_{13}},$
$t_{i_{14}},t_{i_{15}}\},\ \{t_{i_{16}},t_{i_{17}},t_{i_{18}}\}$
from the vertices of $V_i^0$ whenever $T_i=(u_i,v_i,w_i)$ is not in
$M$.

Since $|M|=n$, $|T\backslash M|=m-n$, we can find $7n+6(m-n)=n+6m=q$
partition sets, each set consists of three vertices which belong to
$\overline{U},\overline{V},\overline{W},$ respectively, and they
induce a connected subgraph.

Conversely, let $V_1, V_2,\ldots, V_q$ be the desired partition of
$V(G)$. In the following, we call a $3$-set $\{u,v,w\}$ \emph{a partition set}, if
there is some $j$ such that $\{u,v,w\}=V_j$. Then we choose $T_i\in
M$ if $\{t_{i_6},t_{i_{12}},t_{i_{18}}\}$ is a partition set.

Now we claim that $|M|=n,$ and whenever $(u,v,w)$ and $(u',v',w')$
are distinct triples in $M$ we have $u\neq u',$ $v\neq v'$ and
$w\neq w'$. Indeed, let $T_i=(u_i,v_i,w_i)$. If
$\{t_{i_6},t_{i_{12}},t_{i_{18}}\}$ is a partition set, then either
$\{t_{i_3},t_{i_{4}},t_{i_{5}}\}$ and $\{u_i,t_{i_1},t_{i_2}\}$ are
the partition sets, or $\{t_{i_1},t_{i_3},t_{i_5}\}$ and
$\{u_i,t_{i_2},t_{i_4}\}$ are the partition sets. In either cases,
$u_i$ must belong to a partition set with the other two elements
belong to $V_i^0$. Similar thing is true for $v_i$ and $w_i$. If
$\{t_{i_6},t_{i_{12}},t_{i_{18}}\}$ is not a partition set, then
$\{t_{i_6},t_{i_{11}},t_{i_{12}}\}$ or $\{t_{i_4},t_{i_5},t_{i_6}\}$
is a partition set. But $\{t_{i_6},t_{i_{11}},t_{i_{12}}\}$ can not
be a partition set. If so, then $\{t_{i_8},t_{i_{10}},v_i\}$,
$\{t_{i_7},t_{i_8},t_{i_9}\}$ or $\{t_{i_7},t_{i_8},v_i\}$ must be a
partition set, and no matter in which cases, $t_{i_9}$ or
$t_{i_{10}}$ can not be in a partition set.
% for the first
%one, $t_{i_7}$ and $t_{i_9}$ can not be in a partition set, for the
%second one, $t_{i_{10}}$ can not be in a partition set.
 Thus only $\{t_{i_4},t_{i_5},t_{i_6}\}$ is a partition set. Similarly,
  $\{t_{i_1},t_{i_2},t_{i_3}\},\ \{t_{i_7},t_{i_8},t_{i_9}\},$
$\{t_{i_{10}},t_{i_{11}},t_{i_{12}}\},\
\{t_{i_{13}},t_{i_{14}},t_{i_{15}}\},\
\{t_{i_{16}},t_{i_{17}},t_{i_{18}}\}$ must be partition sets. Then
$u_i,v_i,w_i$ can not be in any partition sets with some vertices in
$V_i^0$.

If $T_i=(u_i,v_i,w_i)$ and $T_j=(u_j,v_j,w_j)$ are distinct triples
in $M$, then $u_i$ is in a partition set with the other two elements
in $V_i^0$, and $u_j$ is in a partition set with the other two
elements in $V_j^0$, and thus $u_i\neq u_j$, and so do $v_i\neq v_j$
and $w_i\neq w_j$. Assume that $|M|=\ell$. If $T_i\in M$, then there
are $7$ partition sets in $V_i^0\cup T_i$.  If $T_i\notin M$, then
there are $6$ partition sets in $V_i^0$. Since
$7\ell+6(m-\ell)=q=n+6m$, $\ell=n$, that is $|M|=n$. The proof is
now complete.
\end{proof}

Now we prove that Conjecture $1$ is true by reducing Problem $1$ to
it.

\begin{theorem}\label{th5}
Given a graph $G$, a $3$-subset $S$ of $V(G)$ and an integer $\ell$ $(2\leq
\ell\leq n-2)$, the problem of deciding whether $G$ contains $\ell$
internally disjoint trees connecting $S$ is $\mathcal {N}\mathcal
{P}$-complete.
\end{theorem}

\begin{proof}
It is easy to see that this problem is in $\mathcal {N}\mathcal{P}$.

Let $G$ be a tripartite graph with partition
$(\overline{U},\overline{V},\overline{W})$ and
$|\overline{U}|=|\overline{V}|=|\overline{W}|=q$. We will construct
a graph $G'$, and a $3$-subset $S$ and an integer $\ell$ such that
there are $\ell$ internally disjoint trees connecting $S$ in $G'$ if
and only if $G$ contains a partition of $V(G)$ into $q$ disjoint
sets $V_1,V_2,\ldots, V_q$ each having three vertices, such that
every $V_i=\{v_{i_1},v_{i_2},v_{i_3}\}$ satisfies that $v_{i_1}\in
\overline{U} $, $v_{i_2}\in \overline{V}$, $v_{i_3}\in \overline{W}
$, and $G[V_i]$ is connected.

We define $G'$ as follows:

$\bullet$ $V(G')=V(G)\cup \{a,b,c\}$;

$\bullet$ $E(G')=E(G)\cup \{au:u\in \overline{U}\}\cup \{bv:v\in \overline{V}\}\cup
\{cw:w\in\overline{W}\}$.

Then we define $S=\{a,b,c\}$, and $\ell=q$.

If there are $q$ internally disjoint trees connecting $S$ in $G'$,
then, since $a,b$ and $c$ all have degree $q$, each tree contains a
vertex from $\overline{U}$, a vertex from $\overline{V}$ and a
vertex from $\overline{W}$, and they induce a connected subgraph.
Since these $q$ trees are internally disjoint, they form a partition
of $V(G)$.

Conversely, if $V_1,V_2, \ldots, V_q$ is a partition of $V(G)$ each
having three vertices, such that every
$V_i=\{v_{i_1},v_{i_2},v_{i_3}\}$ satisfies that  $v_{i_1}\in
\overline{U} $, $v_{i_2}\in \overline{V}$, $v_{i_3}\in \overline{W}
$, and $G[V_i]$ is connected, then let $T_i$ be the spanning tree of
$G[V_i]$ together with the edges $av_{i_1}, bv_{i_2},cv_{i_3}$,
where $V_i=\{v_{i_1},v_{i_2},v_{i_3}\}$. It is easy to check that
$T_1,T_2,\ldots,T_q$ are the desired internally disjoint trees
connecting $S$.
\end{proof}

Now from Theorem \ref{th3} and Theorem \ref{th5}, if $k\geq 3$ is a
fixed integer and $\ell$ is not a fixed integer, the problem of
deciding whether $\kappa(S)\geq \ell$ is $\mathcal
{N}\mathcal{P}$-complete. Since $\kappa_k(G)= min\{\kappa(S) :
S\subseteq V(G) , |S|=k\},$ we have that the problem of deciding
whether $\kappa_k(G)\geq \ell$ is as hard as the problem of deciding
whether $\kappa(S)\geq \ell$. Moreover, the problem of deciding
whether $\kappa_k(G)\geq \ell$ is in $\mathcal {N}\mathcal{P}$, and
so it is $\mathcal {N}\mathcal{P}$-complete. This shows that
Conjecture 2 is true.

\begin{theorem} \label{th6}
For a fixed integer $k\geq 3$, given a graph $G$ and an integer $\ell$
$2\leq \ell\leq n-2$, the problem of deciding whether
$\kappa_k(G)\geq \ell$, is $\mathcal {N}\mathcal {P}$-complete.
\end{theorem}

\section{Hardness results on generalized edge-connectivity}

In this section we consider the computational complexity of the
generalized edge-connectivity $\lambda_k(G)$. Since
$\lambda_k(G)=min\{\lambda(S): S\subseteq V(G), |S|=k\}$, we first
consider $\lambda(S)$, and get the following result.

\begin{theorem}\label{th7}
Given two fixed positive integers $k$ and $\ell$, for any graph $G$
the problem of deciding whether $\lambda_k(G)\geq \ell$ can be
solved by a polynomial-time algorithm.
\end{theorem}

\begin{proof}
Given a connected graph $G$ of order $n$ and a $k$-subset $S$ of
$V(G)$. Let $V(G)=\{v_1,v_2,\ldots,v_n\}$ and $E(G)=\{e_1,e_2,\ldots,e_m\}$.
We will construct a graph $G'$ and a $k$-subset $S'$ of $V(G')$ such that there are
$\ell$ edge-disjoint trees connecting $S$ in $G$ if and only if there are $\ell$ internally
disjoint trees connecting $S'$ in $G'$.

We define $G'$ as follows:

$\bullet$ $V(G')=V(G)\cup
V(L(G))=\{v_1,v_2,\ldots,v_n,e_1,e_2,\ldots,e_m\}$;

$\bullet$ $E(G')=\{e_ie_j:e_ie_j\in E(L(G))\}\cup \{v_ie_j:e_j~is~incident~to~v_i~in~G\}$;\\
where $L(G)$ is the line graph of $G$. We define $S'=S$.

If there are $\ell$ edge-disjoint trees connecting $S$ in $G$, say
$T_1,T_2,\ldots,T_{\ell}$. First for each tree $T_i \ (1\leq i\leq
\ell)$, we replace every edge $e_j=v_{j_1}v_{j_2}$ by a path
$v_{j_1}e_jv_{j_2}$. The obtained graph $T_i'$ now is a tree in
$G'$. Clearly, the trees $T_1',T_2',\ldots,T_{\ell}'$ are $\ell$
edge-disjoint trees connecting $S'=S$ in $G'$. Consider the tree
$T_i'$. If there is a vertex $v\in V(T_i)\backslash S$ such that its
neighbors in $T_i'$ are $e_{i_1},e_{i_2},\ldots,e_{i_p}$, then we
delete the vertex $v$ and its incident edges
$ve_{i_1},ve_{i_2},\ldots,ve_{i_p}$, add a path
$e_{i_1}e_{i_2}\ldots e_{i_p}$. We do this operation for all the
vertex $v\in V(T_i)\backslash S$ for $1\leq i\leq \ell$. The
resulting trees are denoted by $T_1'',T_2'',\ldots, T_{\ell}''$. It
is easy to check that they are internally disjoint trees connecting
$S'$ in $G'$.

Conversely, if there are $\ell$ internally disjoint trees
$T_1,T_2,\ldots,T_{\ell}$ connecting $S'$ in $G'$. Consider any tree
$T_i \ (1\leq i\leq \ell)$. If there is  an edge $e_{j_1}e_{j_2}$ in
$E(T_i)$, by the definition of $E(G')$, $e_{j_1}$ and $e_{j_2}$ are
adjacent in $G$ and hence they have a common vertex $v_j$ in $G$.
Note that $v_je_{j_1}, v_je_{j_2}$ are also edges of $G'$. Thus we
replace the edge $e_{j_1}e_{j_2}$ by a path $e_{j_1}v_je_{j_2}$. We
do this for all the edges of this type in $T_i$. The resulting
connected graph is denoted by $G_i$. Now there is no such edge
$e_ie_j$ in $G_i$ and $d_{G_i}(e) \leq 2.$ If $d_{G_i}(e) = 1,$ we
just delete it. If $d_{G_i}(e) = 2,$ there are two vertices
$v_{j_1}$ and $v_{j_2}$ adjacent with $e$, where $v_{j_1}$ and
$v_{j_2}$ are the endpoints of $e$ in $G$, so we delete the vertex
$e$ and add an edge $v_{j_1}v_{j_2}$, then the obtained graph $G_i'$
is a connected graph with $S \subseteq V(G_i')$ . Let $T_i'$ be the
spanning tree of $G_i'$. It is easy to check that $T_1',T_2',\ldots,
T_{\ell}'$ are $\ell$ edge-disjoint trees connecting $S$ in $G$.

%Conversely, if there are $\ell$ internally disjoint trees
%$T_1,T_2,\cdots,T_{\ell}$ connecting $S'$ in $G'$. Consider the tree
%$T_i \ (1\leq i\leq \ell)$, if some edge $e_{j_1}e_{j_2}\in E(T_i)$,
%by the definition of $E(G')$, $e_{j_1}$ and $e_{j_2}$ are adjacent
%in $G$ and hence they have a common vertex $v$ in $G$. Note that
%$ve_{j_1}, ve_{j_2}$ are also the edges of $G'$. Thus we replace the
%edge $e_{j_1}e_{j_2}$ by a path $e_{j_1}ve_{j_2}$ in $T_i'$. In
%fact, if for each $e_{j_1}\in V(T_i)$ there are $r$ edges in $T_i$,
%say $e_{j_2},e_{j_3},\cdots,e_{j_{r+1}}$, such that $e_{j_1}e_{j_i}\in{E(T_i)}
%\ (2\leq i\leq r+1)$, then $e_{j_1}$ and $e_{j_i}$ are adjacent in
%$G$ and hence they have a common vertex $v_{j_i}$ in $G$. Note that
%$v_{j_2},v_{j_3},\cdots,v_{j_{r+1}}$ are not necessarily different.
%We do this for all the edges of this type. The resulting tree is
%denoted by $T_i' \ (1\leq i\leq \ell)$. For a vertex $e$ in $T_i'$,
%there are two vertices $v_{j_1}$ and $v_{j_2}$ adjacent with $e$,
%where $v_{j_1}$ and $v_{j_2}$ are the endpoints of $e$ in $G$, so we
%delete the vertex $e$ and add an edge $v_{j_1}v_{j_2}$, then the
%obtained graph $T_i''$ is a tree connecting $S$ in $G$. It is easy
%to check that $T_1'',T_2'',\ldots, T_{\ell}''$ are $\ell$
%edge-disjoint trees connecting $S$ in $G$.

From the above reduction, if we want to know whether there are
$\ell$ edge-disjoint trees connecting $S$ in $G$, we can construct a
graph $G'$, and decide whether there are $\ell$ internally disjoint
trees connecting $S'$ in $G'$. By Theorem \ref{th2}, since $k$ and
$\ell$ are fixed, the problem of deciding whether there are $\ell$
internally disjoint trees connecting $S$ can be solved by a
polynomial-time algorithm. Therefore, the problem of deciding
whether there are $\ell$ edge-disjoint trees connecting $S$ can be
solved by a polynomial-time algorithm. The proof is complete.
\end{proof}

Now we consider the problem of deciding whether $\lambda_k(G)\geq
\ell$, for $k\geq3$ a fixed integer but $\ell$ a not fixed integer.
At first, we denote the case when $k=3$ by Problem 2.

\noindent\textbf{Problem 2:} Given a graph $G$, a $3$-subset $S$ of
$V(G)$, and an integer $\ell \ (2\leq \ell \leq n-2)$, decide
whether there are $\ell$ edge-disjoint trees connecting $S$, that
is, $\lambda(S)\geq \ell$ ?

Notice that the reduction from Problem $1$ to the problem in Theorem
\ref{th5} can also be used to be the reduction from Problem $1$ to
Problem $2$ since the $q$ internally disjoint trees connecting $S$
in $G'$ are also $q$ edge-disjoint trees connecting $S$ in $G'$. On
the other hand, if there are $q$ edge-disjoint trees connecting $S$
in $G'$, since the degrees of $a,b$ and $c$ are $q$ we have that
each tree $T_i$ contains one vertex $u_i$ in $\overline{U}$, one vertex $v_i$
in $\overline{V}$, and one vertex $w_i$ in $\overline{W}$, then $\{u_i, v_i, w_i\}\
(1\leq i\leq q)$ constitute a partition of $V(G)$, and each induces
a connected subgraph. Thus the following lemma holds.

\begin{lemma}\label{lem2}
Problem 2 is $\mathcal {N}\mathcal{P}$-complete.
\end{lemma}

Now we show that for a fixed integer $k \geq 4$, replacing the
$3$-subset of $V(G)$ with a $k$-subset of $V(G)$ in Problem 2, the
problem is still $\mathcal {N}\mathcal{P}$-complete, which can be
proved by reducing Problem 2 to it.

\begin{lemma}\label{lem3}
For any fixed integer $k\geq 4$, given a graph $G$, a $k$-subset $S$
of $V(G)$, and an integer
$ \ell \ (2\leq \ell\leq n-2)$, deciding whether
there are $\ell$ edge-disjoint trees connecting $S$, namely deciding
whether $\lambda(S)\geq \ell$, is $\mathcal {N}\mathcal
{P}$-complete.
\end{lemma}

\proof Clearly, the problem is in
$\mathcal {N}\mathcal{P}$.

Given a graph $G$, a $3$-subset $S=\{v_1, v_2, v_3\}$ of $V(G)$ and
an integer $\ell \ (2\leq \ell \leq n-2)$, we construct a graph
$G'=(V',E')$ and a $k$-subset $S'$ of $V(G')$ and let $\ell'=\ell$
such that there are $\ell$ edge-disjoint trees connecting $S$ in $G$
if and only if there are $\ell$ edge-disjoint trees connecting $S'$
in $G'$.

We define $G'$ as follows:

$\bullet$ $V(G')=V(G)\cup \{a^i: 1\leq i\leq k-3\}\cup \{a_j^i:
1\leq i\leq k-3, 1\leq j\leq \ell \}.$

$\bullet$ $E(G')=E(G)\cup \{v_1a_j^i:1\leq i\leq k-3, 1\leq j\leq
\ell\}\cup \{a^ia_j^i:1\leq i\leq k-3, 1\leq j\leq \ell\}$.

Let $S'=S\cup\{a^1, a^2,\ldots, a^{k-3}\}$.

It is easy to check that there are $\ell$ edge-disjoint trees
connecting $S$ in $G$ if and only if there are $\ell$ edge-disjoint
trees connecting $S'$ in $G'$. The proof is complete.\qed

From Lemma \ref{lem2} and Lemma \ref{lem3}, we obtain the following
result.

\begin{theorem}\label{th8}
For any fixed integer $k\geq 3$, given a graph $G$, a $k$-subset $S$
of $V(G)$, and an integer
$ \ell \ (2\leq \ell\leq n-2)$, deciding whether
there are $\ell$ edge-disjoint trees connecting $S$, namely deciding
whether $\lambda(S)\geq \ell$, is $\mathcal {N}\mathcal
{P}$-complete.
\end{theorem}

Similar to the argument in the proof of Conjecture 2, we conclude
that if $k\geq 3$ is a fixed integer but $\ell$ is not a fixed
integer, the problem of deciding whether $\lambda_k(G)\geq \ell$ is
$\mathcal {N}\mathcal{P}$-complete. This proves the next result.

\begin{theorem} \label{th9}
For a fixed integer $k\geq 3$, given a graph $G$ and an integer
$ \ell \ (2\leq \ell\leq n-2)$, the problem of deciding whether
$\lambda_k(G)\geq \ell$ is $\mathcal {N}\mathcal {P}$-complete.
\end{theorem}

Now we turn to the case that $\ell$ is a fixed integer but $k$
is not a fixed integer.
At first, we consider the case $\ell=2$, and denote it by Problem 3.

\noindent\textbf{Problem 3:} Given a graph $G$, a subset $S$ of
$V(G)$, decide whether there are two edge-disjoint trees connecting
$S$, that is $\lambda(S)\geq 2$ ?

We show that Problem 3 is $\mathcal {N}\mathcal{P}$-complete by
reducing $3$-SAT to it.

\noindent\textbf{BOOLEAN 3-SATISFIABILITY (3-SAT):} Given a boolean
formula $\phi$ in conjunctive normal form with three literals per
clause, decide whether $\phi$ is satisfiable ?

\begin{lemma}\label{lem5}
 Problem 3 is $\mathcal {N}\mathcal{P}$-complete.
\end{lemma}

\begin{proof}
Clearly, Problem 3 is in $\mathcal {N}\mathcal{P}$.

Let $\phi$ be an instance of $3$-SAT with clauses $c_1,c_2,\ldots,
c_m$ and variables $x_1,x_2,\ldots,x_n$. We construct a graph
$G_\phi=(V_\phi, E_\phi)$ and define a subset $S$ of $V(G_\phi)$
such that there are two edge-disjoint trees connecting $S$ if and
only if $\phi$ is satisfiable.

We define $G_\phi$ as follows: \noindent{\itshape}
\begin{eqnarray*}
\bullet~~~V(G_\phi)&=&\{\hat{x_i}, x_i, \overline{x}_i: 1\leq i\leq
n\}\cup \{c_i,c_i':1\leq i\leq m\}\cup \{a, b\};\\
\bullet~~~E(G_\phi)&=&\{\hat{x_i}x_i,\hat{x_i}\overline{x}_i:1\leq
i\leq n\}
\cup\{x_ic_j:x_i\in c_j\}\cup\{\overline{x}_ic_j:\overline{x}_i\in c_j\} \\
& &\cup\{x_1x_i, x_1\overline{x}_i:2\leq i\leq
n\}\cup\{\overline{x}_1x_i, \overline{x}_1\overline{x}_i:2\leq i\leq
n\}\cup \{ab\}~~~\\
& &\cup\{ac_i', c_ic_i': 1\leq i\leq m\}\cup\{bx_i, b\overline{x}_i:1\leq i\leq n\}.
\end{eqnarray*}

Now we define $S=\{c_i':1\leq i\leq m\}\cup \{\hat{x_i}:1\leq i\leq
n\}.$

Suppose that there are two edge-disjoint trees $T_1$ and $T_2$
connecting $S$. We know that the edge $ab$ can not be in both
$E(T_1)$ and $E(T_2)$, and so assume that $T_1$ does not contain
$ab$. Next we claim that $ac_i'\notin E(T_1)$ for $1\leq i\leq m.$
For otherwise, without loss of generality, let $ac_1'\in E(T_1)$.
Since the degree of $c_i'$ is 2, $T_1$ contains only one of the
edges $ac_i'$ and $c_i'c_i$ for $1\leq i\leq m.$ So $c_1'c_1\notin
E(T_1)$. Because $T_1$ is connected, the path from $c_1'$ to
$\hat{x_j}$ in $T_1$ must contain the edges $ac_1', ac_k',c_k'c_k$
for some $1\leq k\leq m$ and a path from $c_k$ to $\hat{x_j}$. But in this
case, the degree of $c_k'$ in $T_1$ is 2, a contradiction.
Therefore, $T_1$ contains all the edges $c_ic_i' \ (1\leq i\leq m)$,
and for each $c_i \ (1\leq i\leq m)$, there must exist some $x_j\in
V(T_i)$ such that $c_ix_j\in E(T_1)$ or $\overline{x}_j\in V(T_i)$
such that $c_i\overline{x}_j\in E(T_1)$. As $\hat{x_i}\in S \ (1\leq
i\leq n)$ and the degree of $\hat{x_i}$ is 2, $V(T_1)$ contains only
one of the neighbors of $\hat{x_i}$. If $x_i$ is contained in
$V(T_i)$, then set $x_i=1$. Otherwise, set $x_i=0$. Clearly, we
conclude that $\phi$ is satisfiable in this assignment.

On the other hand, suppose that $\phi$ is satisfiable with the
assignment $t$. We will find two edge-disjoint trees connecting $S$
as follows.

For each $1\leq i\leq m$, there must exist a $j \ (1\leq j\leq n)$
such that $x_j\in c_i$ and $t(x_j)=1$, or $\overline{x}_j\in c_i$
and $t(x_j)=0$. We then construct $T_1$ with edge set $\{c_ix_j(or\
c_i\overline{x}_j), c_ic_i':1\leq i\leq m\}.$ Obviously, $V(T_1)$
can not contain both $x_i$ and $\overline{x}_i$. If none of $x_i$
and $\overline{x}_i$ is in $V(T_i)$, we choose any one of them
belonging to $V(T_1)$. Now if $x_1\in V(T_1)$, we add $x_1x_i$ (if
$x_i\in V(T_1)$) or $x_1\overline{x}_i$ (if $\overline{x}_i\in
V(T_1)$) to $E(T_1)$, for $2\leq i\leq n$. Otherwise, if
$\overline{x}_1\in V(T_1)$, we add $\overline{x}_1x_i$ (if $x_i\in
V(T_1)$) or $\overline{x}_1\overline{x}_i$ (if $\overline{x}_i\in
V(T_1)$) to $E(T_1)$, for $2\leq i\leq n$. Finally, if $x_i\in
V(T_1)$, we add $\hat{x_i}x_i$ to $E(T_1)$, if $\overline{x}_i\in
V(T_1)$, we add $\hat{x_i}\overline{x}_i$ to $E(T_1)$. It is easy to
check that the graph $T_1$ is indeed a tree containing $S$. Now let
$T_2$ be a tree containing $ab$, $ac_i'$ for $1\leq i\leq m$, $bx_j$
and $\hat{x_j}x_j$ (if $\overline{x}_j\in V(T_1)$),
$b\overline{x}_j$ and $\hat{x_j}\overline{x}_j$ (if $x_j\in
V(T_1)$). Then we conclude that $T_1$ and $T_2$ are two
edge-disjoint trees connecting $S$. The proof is complete.
\end{proof}

Now we show that for a fixed integer $\ell \geq 3$, the problem is still
$\mathcal {N}\mathcal{P}$-complete, which can be proved by reducing Problem 3 to it.

\begin{lemma}\label{lem4}
For any fixed integer $\ell\geq 3$, given a graph $G$, a subset $S$
of $V(G)$, deciding whether there are $\ell$ edge-disjoint trees
connecting $S$, namely deciding whether $\lambda(S)\geq \ell$, is
$\mathcal {N}\mathcal {P}$-complete.
\end{lemma}

\proof  Clearly, the problem is in $\mathcal {N}\mathcal{P}$. Let
$G$ and a subset $S$ of $V(G)$ be a instance of Problem 3. We will
construct a graph $G'=(V',E')$ and a subset $S'$ of $V(G')$ such
that there are two edge-disjoint trees connecting $S$ if and only if
there are $\ell$ edge-disjoint trees connecting $S'$.

Assume that $S=\{v_1,v_2, \ldots, v_k\}$. To define $G'$, we add $k$
new vertices $\{v_1',v_2', \ldots, v_k'\}$. Then for every $v_i'$,
we add two paths $v_iv_i^1v_i'$ and $v_iv_i^2v_i'$ for $1\leq i\leq
k$. Finally we add $\ell-2$ new vertices $\{a_1, a_2, \ldots,
a_{\ell-2}\}$, and join each $a_i$ to $v_1', v_2', \ldots, v_k'$ for
$1\leq i \leq \ell-2$. That is,

$\bullet$ $V(G')=V(G)\cup\{v_i':1\leq i\leq k\}\cup \{v_i^1,v_i^2:
1\leq i\leq k\}\cup \{a_j: 1\leq j\leq \ell-2\};$

$\bullet$ $E(G')=E(G)\cup\{v_iv_i^1, v_i^1v_i': 1\leq i\leq k\}
\cup\{v_iv_i^2, v_i^2v_i': 1\leq i\leq k\}\cup \{a_iv_j':
1\leq i\leq \ell-2, 1\leq j\leq k\}.$

We define $S'=\{v_1',v_2', \ldots, v_k'\}$.

Suppose that there are two edge-disjoint trees $T_1$ and $T_2$
connecting $S$ in $G$, let $T_1'$ be the tree containing $T_1$ and
the paths $v_iv_i^1v_i' \ (1\leq i\leq k)$, $T_2'$ be the tree
containing $T_2$ and the paths $v_iv_i^2v_i' \ (1\leq i\leq k)$.
Then $T_1'$ and $T_2'$ are two edge-disjoint trees connecting $S'$
in $G'$. For each $a_i \ (1\leq i\leq \ell-2)$, let $T'_{i+2}$ be
the tree with the edges $a_iv_j' \ (1\leq j\leq
k)$. Then we find $\ell$ edge-disjoint trees connecting $S'$ in
$G'$.

Conversely, suppose that there are $\ell$ edge-disjoint trees $T_1',
T_2',\ldots, T_{\ell}'$ connecting $S'$ in $G'$. Since the degree of
$v_i'$ in $G'$ is $\ell$, $v_1',v_2',\ldots,v_k'$ are all leaves in
each $T_j'$. If the tree $T_j'$ contains the edge $a_jv_i'$, then
$T_j'$ is the star with center vertex $a_j$ and leaves $v_i' \
(1\leq i\leq k).$ Otherwise, there is some vertex in $S'$ with its
degree at least two, a contradiction. Thus, among $\{T_1', T_2',\ldots, T_{\ell}'\}$,
there are $\ell-2$ trees connecting $S'$, each is a star with center vertices $a_j \
(1\leq j\leq \ell-2).$ The remaining two trees must contain the
paths $v_i'v_i^1v_i$ or $v_i'v_i^2v_i$ for $1\leq i\leq k$ and two
$S$-trees in $G$. Therefore, we find two edge-disjoint trees
connecting $S$ in $G$. The proof is complete.\qed

Combining Lemma \ref{lem5} with Lemma \ref{lem4}, we obtain the
following result.

\begin{theorem}\label{th10}
For any fixed integer $\ell\geq 2$, given a graph $G$, a subset $S$
of $V(G)$, deciding whether there are $\ell$ edge-disjoint trees
connecting $S$, namely deciding whether $\lambda(S)\geq \ell$, is
$\mathcal {N}\mathcal {P}$-complete.
\end{theorem}

\end{document}